\documentclass[11pt]{amsart}
\usepackage{amssymb,a4wide}
\usepackage{hyperref}
\date{}
\newtheorem{theorem}{Theorem}%[section]
\newtheorem{lemma}[theorem]{Lemma}%[section]
\newtheorem{corollary}[theorem]{Corollary}%[section]
\newtheorem{proposition}[theorem]{Proposition}%[section]

\newtheorem{problem}[theorem]{Problem}
\newtheorem{example}[theorem]{Example}
\theoremstyle{definition}
\newtheorem{remark}[theorem]{Remark}%[section]

\newcommand{\suc}{\mathrm{succ}}
\newcommand{\pre}{\mathrm{prec}}
\def\phi{\varphi}

\def\int{\operatorname{int}}
\newcommand{\IZ}{\mathbb Z}
\newcommand{\IN}{\mathbb N}
\newcommand{\IR}{\mathbb R}

\begin{document}
\title[Zero-sum subsets of decomposable sets in Abelian groups]{Zero-sum subsets of decomposable sets\\ in Abelian groups}
\author{Taras Banakh and Alex Ravsky}
\address{T.~Banakh: Ivan Franko National University of Lviv (Ukraine) and Jan Kochanowski University in Kielce (Poland)}
\email{t.o.banakh@gmail.com}
\address{A.~Ravsky: Department of Analysis, Geometry and Topology,
Pidstryhach Institute for Applied Problems of Mechanics and Mathematics
National Academy of Sciences of Ukraine}%,Naukova 3-b, Lviv, 79060, Ukraine}
\email{alexander.ravsky@uni-wuerzburg.de}
\subjclass{05E15}
\keywords{decomposable set, abelian group, sum-set}

\keywords{}
%\subjclass{22A15, 54H99,54H11}%MZC 2000
\begin{abstract} A subset $D$ of an abelian group is {\em decomposable} if $ \emptyset\ne D\subset D+D$. In the paper we give partial answers to an open problem asking whether every finite decomposable subset $D$ of an abelian group contains a non-empty subset $Z\subset D$ with $\sum Z=0$. For every $n\in\IN$ we present a decomposable subset $D$ of cardinality $|D|=n$ in the cyclic group of order $2^n-1$ such that $\sum D=0$, but $\sum T\ne 0$ for any proper non-empty subset $T\subset D$. On the other hand, we  prove that every decomposable subset $D\subset\IR$ of cardinality $|D|\le 7$ contains a non-empty subset $T\subset D$ of cardinality $|Z|\le\frac12|D|$ with $\sum Z=0$. For every $n\in\IN$ we present a subset $D\subset\IZ$ of cardinality $|D|=2n$ such that $\sum Z=0$ for some subset $Z\subset D$ of cardinality $|Z|=n$ and $\sum T\ne 0$ for any  non-empty subset $T\subset D$ of cardinality $|T|<n=\frac12|D|$. Also we prove that every finite decomposable subset $D$ of an Abelian group contains two non-empty subsets $A,B$ such that $\sum A+\sum B=0$.
\end{abstract}
\maketitle \baselineskip15pt

\section{Introduction}

A subset $D$ of an Abelian group $G$ is called \emph{decomposable} if $\emptyset\ne D\subset D+D:=\{x+y:x,y\in D\}$. For a finite subset $F$ of an Abelian group put $\sum F=\sum_{x\in F}x$.  In this paper we discuss the following open problem, whose special case for the additive group of real numbers was posed by Gjergji Zaimi in 2010 on {\tt MathOverflow} \cite{Z}.

\begin{problem}\label{prob} Let $D$ be a finite decomposable subset of an Abelian group $G$. Is $\sum Z=0$ for some non-empty set $Z\subset D$?
\end{problem}

Observe that the answer to Problem~\ref{prob} is affirmative if the group $G$ is {\em Boolean} (which means that $x+x=0$ for any $x\in G$). Indeed, take any element $a\in D$ and find elements $b,c\in D$ with $a=b+c$. If $0\in\{a,b,c\}$, then $T=\{0\}\subset\{a,b,c\}\subset D$ is a subset with $\sum T=0$. If $0\notin \{a,b,c\}$, then the elements $a,b,c$ are pairwise distinct and the set $T=\{a,b,c\}$ has $\sum T=a+b+c=2(b+c)=0$. Therefore, any decomposable set $D$ in a Boolean group contains a subset $T\subset D$ of cardinality $|T|\le 3$ with $\sum T=0$.

This simple upper bound does not hold for decomposable sets in arbitrary groups.

\begin{example}\label{ex1} Let $n\ge 2$ and $G$ be a cyclic group of order $2^{n}-1$ with generator $g$. The set $D=\{2^kg:0\le k<n\}$ is decomposable and has $\sum D=0$ but $\sum T\ne\emptyset$ for any subset $T\subset D$ of cardinality $0<|T|<|D|=n$.
\end{example}

\begin{example}\label{ex2} For every $n\in\IN$ the subset $$D_{n}:=\{2^k:0\le k<n\}\cup \{2^k-2^n+1:0\le k<n\}$$ of cardinality $|D_n|=2n$ in the infinite cyclic group $\IZ$ is decomposable and the subset $$Z=\{2^k:1\le k<n\}\cup\{2-2^n\}\subset D_{n}$$ of cardinality $|Z|=n$ has $\sum Z=0$. On the other hand, $\sum T\ne0$ for every non-empty subset $T\subset D_{n}$ of cardinality $|T|<n$.
\end{example}

The properties of the set $D_{n}$ from Example~\ref{ex2} will be established in Section~\ref{s:ex}.

The above examples suggest to assign to every finite subset $D$ of an Abelian group $G$ the largest number $z(D)\le |D|+1$ such that $\sum T\ne 0$ for any non-empty subset $T\subset D$ of cardinality $|T|<z(D)$.   Therefore, $z(D)$ is equal to the smallest cardinality of a subset $Z\subset D$ with $\sum Z=0$ if such set $Z$ exists and $z(D)=|D|+1$ in the opposite case.

In terms of the number $z(D)$, Problem~\ref{prob} can be reformulated as follows.

\begin{problem}\label{prob4} Let $D$ be a finite decomposable subset of an Abelian group. Is $z(D)\le |D|$?
\end{problem}

The decomposable set $D_n$ from Example~\ref{ex1} has $z(D)=|D|=n$ and the decomposable set $D_n\subset\IZ$ from Example~\ref{ex2} has $z(D_n)=n=\frac12|D_n|$.

This suggests the following refinement of Problems~\ref{prob} and \ref{prob4} for the infinite cyclic group $\IZ$.

\begin{problem}\label{prob:z} Is $z(D)\le \frac12|D|$ for any finite decomposable subset $D\subset \IR$?
\end{problem}

A special case of this problem was posed by Aryabhata in \cite{A}.

\begin{problem}[Aryabhata] Is $z(D)\le 7$ for any  decomposable subset $D\subset \IZ$ of cardinality $|D|=15$?
\end{problem}

In Section~\ref{s:p7} we shall provide an affirmative answer to Problem~\ref{prob:z} for decomposable subsets $D\subset \IZ$ of cardinality $|D|\le 7$.

\begin{proposition}\label{p7} Any decomposable subset $D\subset\IR$ of cardinality $|D|\le 7$ has $z(D)\le\frac12|D|$.
\end{proposition}

We also observe that for $n\in\{2,3\}$ the decomposable set $D_n=\{2^k,2^k-2^n+1:0\le k<n\}$ form Example~\ref{ex2} is unique in the following sense.

\begin{proposition}\label{p8} Every decomposable set $D\subset \mathbb R$ with $n=z(D)=\frac12|D|\in\{2,3\}$ is equal to $D_n{\cdot}r$ for some real number $r$.
\end{proposition}

\begin{problem} Is every finite decomposable set $D\subset \IR$ with $z(D)=\frac12|D|=n\ge 2$ equal to $D_n{\cdot}r$ for some  real number $r$?
\end{problem}

The following proposition shows that the problems on fiinite decomposable sets in torsion-free Abelian groups can be reduced to the case of the infinite cyclic group.

\begin{proposition} For any finite decomposable set $D$ is a torsion-free Abelian group $G$ there exists a  decomposable set $D'\subset\IZ$ such that $|D'|=|D|$ and $z(D')=z(D)$.
\end{proposition}

\begin{proof} We lose no generality assuming that $G$ is finitely-generated and hence is isomorphic to $\IZ^n$ for some $n\in\IN$. Let $e_1,\dots,e_n$ be the standard generators of the group $\IZ^n$. Find $m\in\IN$ such that $\{\sum F:F\subset D\}\subset \{\sum_{i=1}^nx_ie_i:(x_i)_{i=1}^n\in[-m,m]^n\}=[-m,m]^n$. Consider the homomorphism $h:\IZ^n\to\IZ$ such that $h(e_i)=(2m+1)^i$ for all $1\le i\le n$.   It is easy to see that the restriction $h{\restriction}[-m,m]^n$ is injective. Consequently, a subset $F\subset D$ has $\sum F=0$ if and only if $\sum h(F)=0$. This implies that for the set $D'=h(D)$ we have the equalities $|D'|=|D|$ and $z(D')=z(D)$.
\end{proof}

Our final result provides an affirmative answer to a weak version of Problem~\ref{prob}. The following theorem will be proved in Section~\ref{s:tree}.

\begin{theorem}\label{t:main} For any finite decomposable subset $D$ of an Abelian group there are two non-empty sets $A,B\subset D$ such that $\sum A+\sum B=0$.
\end{theorem}

\begin{corollary}\label{c:main} For any finite decomposable subset $D$ of an Abelian group, there exists a non-empty subset $T\subset D$ and a function $f:T\to\{1,2\}$ such that $\sum_{x\in T}f(x)\cdot x=0$.
\end{corollary}

A decomposable subset $D$ of an Abelian group is called {\em minimal decomposed} if no proper subset of $D$ is decomposed. It is clear that every finite decomposed set contains a minimal decomposed set.

Corollary~\ref{c:main} can be compared with the following result that was essentially proved by Hsien-Chih Chang \cite{HC} in his answer to the problem of Zaimi \cite{Z}.

\begin{proposition} For every finite minimal decomposable subset $D$ of an Abelian group there exists a function $f:D\to\omega$ such that $\sum_{x\in D}f(x)=|D|$ and $\sum_{x\in D}f(x)\cdot x=0$.
\end{proposition}

\begin{proof}  By the decomposability of $D$, there exist functions $\alpha,\beta:D\to D$ such that $x=\alpha(x)+\beta(x)$ for every $x\in X$. For every $x\in D$ let $g(x)=|\alpha^{-1}(x)|+|\beta^{-1}(x)|$ and observe that $\sum_{x\in D}g(x)=|D|+|D|$. The minimal decomposability of $D$ ensures that $D=\alpha(D)\cup\beta(D)$ and hence $f(x):=g(x)-1\ge 0$ for every $x\in D$. Then $\sum_{x\in D}f(x)=(|D|+|D|)-|D|=|D|$.
It follows that $\sum D=\sum_{x\in D}(\alpha(x)+\beta(x))=\sum_{y\in D}g(y)\cdot y$ and hence $\sum_{x\in D}f(x)\cdot x=\sum_{x\in D}(g(x)-1)\cdot x=0$.
\end{proof}

\section{Properties of the decomposable set in Example~\ref{ex2}}\label{s:ex}

Given a natural number $n$ consider the subsets $A=\{1,2,4,\dots, 2^{n-1}\}$ and $$D_n:=A\cup (A-(2^n-1))$$ in the group $\IZ$ of integers. Then $|D_n|=2n$.

The set $D_n$ is decomposable, because, clearly, each element of the set $A\setminus\{1\}$ is decomposable,
$1=2^{n-1}+(2^{n-1}-(2^n-1))$, $2^k-(2^n-1)=2^{k-1}+(2^{k-1}-(2^n-1))$ for every positive $k<n$,
and $1-(2^n-1)=(2^{n-1}-(2^n-1))+(2^{n-1}-(2^n-1))$.
\smallskip

It is clear that the set $Z=\{2^k:1\le k<n\}\cup\{2^n-2\}$ has cardinality $|Z|=n=\frac12|D_n|$ and $\sum Z=0$.
\smallskip

Next, we prove that every subset $T\subset D_n$ of cardinality $|T|<n$ has $\sum T\ne 0$. Assuming that $\sum T=0$, we conclude that $n\ge 3$ and $T$ contains at most $n-2$ positive elements. Since all of
them are distinct elements of $A$, their sum is at most $2^n-4$. On the other hand, the largest negative element of the set $A-(2^n-1)$ is $-2^{n-1}+1=-(2^n-2)/2$. Thus if $T$ contains at least
two negative elements then $\sum T<0$. If $T$ contains exactly one negative element $2^k-2^n+1$ then $\sum T=0$
implies that we have a representation of $2^n-1$ as a sum of at most $n-1$ powers of $2$ with at most one power used twice.
This representation collapses to a sum of at most $n-1$ distinct powers of $2$.
%not bigger than $2^{k-1}$,
If the representation contains a power $2^l$ with $l\ge n$ then it is bigger than $2^n-1$.
Otherwise the sum is at most $2^1+2^2+\dots +2^{n-1}=2^n-2<2^k-1$. Thus $z(D_n)\ge n$.

\section{Proof of Propositions~\ref{p7} and \ref{p8}}\label{s:p7}

We divide the proof of Propositions~\ref{p7} and \ref{p8} into five lemmas. In fact, Proposition~\ref{p8} follows from Lemmas~\ref{l4} and \ref{l6} proved below.

\begin{lemma}\label{l3} Every decomposable subset $D\subset \IR$ of cardinality $|D|\le 3$ contains zero and hence has $z(D)=1\le \frac12|D|$.
\end{lemma}

\begin{proof} To derive a contradiction, assume that $0\notin D$. Replacing $D$ by $-D$, if necessary, we can assume that $D$ contains a unique positive element $p$.  Since $p=a+b$ for some
elements $a,b\in D$, one of the numbers $a$ or $b$ is positive, so it equals $p$ and the other summand is zero.
\end{proof}

The following lemma was proved by  Ingdas~\cite{Z}. We present a short proof for convenience of the reader.

\begin{lemma}[Ingdas]\label{l4} Any decomposable subset $D\subset\IR$ with $|D|=4$ and $z(D)\ge 2$ is equal to the set $\{-2,-1,1,2\}{\cdot}r=D_2{\cdot}r$ for some positive  real number $r$.
\end{lemma}

\begin{proof} Since $z(D)\ge 2$, the set $D$ does not contain zero. Then $D$ should contains at least two positive numbers and at least two negative numbers (otherwise $D$ will be not decomposable). Since $|D|=4$, the set $D$ contains exactly two positive and two negatice numbers. Let $c$ be the largest positive element of $D$. Since $D$ is decomposable, $c=a+b$ for some $a,b\in D$. Since $D$ does not contain zero, the maximality of $c$ ensures that the elements $a,b$ are strictly positive and hence coincide with the unique positive element $r$ of the set $D\setminus\{c\}$. Therefore, $c=a+b=2r$. By the same reason, the subset $\{x\in D:x<0\}$ of negative numbers is equal to $\{2n,n\}$ for some negative real number $n$. Write the element $r$ as $r=x+y$ for some $x,y\in D$ with $x\le y$.
Taking into account that $D$ contains no zero and $r,2r$ are unique positive elements of $D$, we conclude that $y=2r$ and then $x=r-y=r-2r=-r\in\{2n,n\}$. If $-r=2n$, then $D=\{-r,-\frac12r,r,2r\}$ is not decomposable as $-\frac12r\notin D+D$. So, $r=-n$ and hence $D=\{-2r,-r,r,2r\}=\{-2,-1,1,2\}{\cdot}r$.
\end{proof}

\begin{lemma}\label{l45} Any decomposable subset $D\subset\IR$ of cardinality $|D|\in\{4,5\}$ contains a subset $T\subset D$ of cardinality $|Z|\in\{1,2\}$ with $\sum Z=0$. Consequently, $z(D)\le 2\le\frac12|D|$.
\end{lemma}

\begin{proof} If $0\in D$, then $Z=\{0\}$ has $\sum Z=0$ and witnesses that $z(D)=0\le \frac12|D|$.

So, assume that $0\notin D$. Replacing $D$ by $-D$ we can assume that $D$ has at most two positive elements. Let $c$ be the largest positive element of $D$. Since $D$ is decomposable, $c=a+b$ for some $a,b\in D$. Since $D$ does not contain zero, the maximality of $c$ ensures that the elements $a,b$ are strictly positive and hence coincide with the unique positive element of the set $D\setminus\{c\}$. Therefore, $c=a+b=2a$. Write the element $a$ as $a=x+y$ for some $x,y\in D$ with $x\le y$.
Taking into account that $D$ contains no zero and $a,c$ are unique positive elements of $D$, we conclude that $y=c$ and then $x=a-y=a-c=a-2a=-a$. Then the set $Z=\{-a,a\}$ has $\sum Z=0$ and witnesses that $z(D)\le 2\le\frac12|D|$.\end{proof}

\begin{remark} It can be shown that each decomposable subset $D\subset \IR$ of cardinality $|D|=5$ with $0\notin D$ is equal to one of the sets
$$\{-3,-2,-1,1,2\}\cdot r,\;\;\{-4,-2,-1,1,2\}\cdot r,\;\;\{-3,-2,-1,2,4\}\cdot r$$ for some non-zero number $r$.
\end{remark}

\begin{lemma}\label{l6} Any decomposable subset $D\subset\IR$ of cardinality $|D|=6$ with $z(D)\ge3$ coincides with $\{1,2,4,-3,-5,-6\}\cdot r$ for some real number $r$ and hence has $z(D)=3=\frac12|D|$.
\end{lemma}

\begin{proof}
Let $D$ be a decomposable set consisting of six real numbers. If $0\in D$ then
$z(D)=1$. If $D$ contains exactly one (resp. two) positive (or negative) elements then
similarly to the case $n\le 3$ (resp. $4\le n\le 5$) we can show that $z(D)=1$ (resp. $z(D)\le 2$).

So it remains to consider the case when $D$ consists of three positive and three negative numbers.
Let $p_{\max}$ (resp. $n_{\min}$) be the largest positive (resp. the smallest negative)
element of $D$. Write $p_{\max}$ as $p_{\max}=p+p'$ for some numbers $p,p'\in D$ with $p\le p'$. By the maximality of $p_{\max}$, both numbers  $p$
and $p'$ are positive. Assume that $p<p'$ and write $p=\bar p+a$, $p'=\bar p'+a'$
for some elements of $D$ with positive $\bar p$ and $\bar p'$. If $\bar p=p_{\max}$
(resp. $\bar p'=p_{\max}$) then $p'+a=0$ (resp. $p+a'=0$), so $z(D)=2$.
Therefore, we can assume that $\bar p\ne p_{\max}\ne \bar p'$. In this case $\bar p=p'$ and $\bar p'=p$. Consequently, $a=p-\bar p=p-p'=\bar p'-p'=-a'$ and $Z=\{a,a'\}$ is a set with $\sum Z=0$, witnessing that $z(D)\le 2$.

Thus it remains to consider the
case when $p_{\max}=2p'$ for some $p'\in D$ and, by the symmetry, $n_{\min}=2n'$ for some $n'\in D$.
Let $p'$,$2p'=p_{\max},p''$ be positive elements of the set $D$ and $n'$, $2n'=n_{\min},n''$ be its negative elements.
Then $n'$ is a sum of two elements of $D$ at least one of which is negative.
If $n'=2n'+x$ for some $x\in D$ then $n'+x=0$ and $z(D)=2$.
Thus $n'=n''+a$ for some $a\in D$.
If $n''=n'+b$ for some $b\in D$ then $a+b=0$ and $z(D)=2$.
Thus $n''=2n'+\bar p$ for some positive $\bar p\in D$.
Similarly, $p'=p''+a'$ for some $a'\in D$ and $p''=2p'+\bar n$ for some negative $\bar n\in D$.
If $a>0$ and $a'<0$ then $\max\{\bar n,a'\}\le n'$ and $p'\le\min\{\bar p,a\}$. So

$$0=n'+a+\bar p\ge n'+p'+p'>p'+n'>p'+n'+n'\ge p'+a'+\bar n=0,$$
a contradiction.
Thus $a<0$ or $a'>0$. If $a<0$, then since $n'=n''+a$ we have $n'=2n''$. Similarly,
if $a'>0$, then $p'=2p''$. Reverting the signs of elements of $D$, if needed, we can suppose that
$p'=2p''$. Then $\{p'',2p'', 4p''\}\in D$, so $\bar n=-3p''$ and $n''=\bar p+2n'$.
The following cases are possible:
\begin{itemize}
\item[1.] $n'=\bar n=-3p'' $, so $2n'=-6p''$.
\item[1.1.] If $\bar p=p'' $ then $n''=-5p''$ so the proposition claim holds.
\item[1.2.] If $\bar p=2p''$ then $n''=-4p''$, so $n''+4p''=0$ and $z(D)=2$.
\item[1.3.] If $\bar p=4p''$ then $n''=-2p''$, so $n''+2p''=0$ and $z(D)=2$.
\item[2.] $2n'=\bar n=-3p''$, so $n'=-1.5p''$.
\item[2.1.] If $\bar p=p'' $ then $n''=-2p''$, so $n''+2p''=0$ and $z(D)=2$.
\item[2.2.] If $\bar p=2p''$ then $n''=- p''$, so $n''+ p''=0$ and $z(D)=2$.
\item[2.3.] If $\bar p=4p''$ then $n''=p''>0$, a contradiction.
\item[3.] $n''=\bar n=-3p''$. Then $\bar p+2n'+3p''=0$.
\item[3.1.] If $\bar p=p'' $ then $n'=-2p''  $, so $n'+2p''=0$ and $z(D)=2$.
\item[3.2.] If $\bar p=2p''$ then $n'=-2.5p''$, so $2n''=-5p''$ and $n'\not\in D+D$, a
contradiction.
\item[3.3.] If $\bar p=4p''$ then $n'=-3.5p''$, so $2n''=-7p''$ and $n'\not\in D+D$, a
contradiction.
\end{itemize}
\end{proof}

\begin{lemma}\label{l7} Every decomposable subset $D\subset\IR$ of cardinality $|D|=7$ has $z(D)\le 3\le\frac12|D|$.
\end{lemma}

\begin{proof}
Let $D$ be a decomposable set consisting of seven real numbers. If $0\in D$ then
$z(D)=1$. If $D$ contains exactly one (resp. two) positive (or negative) elements then
similarly to the case $n\le 3$ (resp. $4\le n\le 5$) we can show that $z(D)=1$ (resp. $z(D)\le 2$).
So, reverting the signs of elements of $D$, if needed, it remains to consider
the case when $D$ consists of four positive and three negative numbers.
Let $p_{\max}$ (resp. $n_{\min}$) be the largest positive (resp. the smallest negative)
element of $D$. Similarly to the proof of Lemma~\ref{l6} we can show that
$n_{\min}=2n'$ for some $n'\in D$.

Let $n'$, $2n'=n_{\min},n''$ be negative elements of the set $D$.
Similarly to the proof of Lemma~\ref{l6} we can show that
$n'=n''+a$ for some $a\in D$ and $n''=2n'+\bar p$ for some positive $\bar p\in D$.
Then $n'+a+\bar p=0$. If these numbers are distinct then $z(D)\le 3$. So we assume the converse.
Since $n'=n''+a$, $a\ne n'$, so $a=\bar p=n'-n''=n''-2n'$ and $3n'=2n''$.
Divide all elements of the set $D$ by $\frac12|n'|$. Then it will have elements $-2$, $-3$, $-4$, and $1$.

Depending on the representation of $1$ as a sum of elements of $D$, the following cases are possible:

\begin{itemize}
\item[1.] $1=3-2$, and $3\in D$. Since $-3\in D$, we have $z(D)\le 2$.
\item[2.] $1=4-3$, and $4\in D$. Since $-4\in D$, we have $z(D)\le 2$.
\item[3.] $1=5-4$, and $5\in D$. Since $-2,-3\in D$, we have $z(D)\le 3$.
\item[4.] $1$ is a sum of distinct positive elements of $D$.
Then $p_{\max}<2$ so a sum of each positive and negative elements of $D$ is negative.
Then the smallest positive element of $D$ does not belong to $D+D$, a contradiction.
\item[5.] $1=0.5+0.5$ and $0.5\in D$. If the smallest positive element $p_{\min}$ of $D$ is less than
$0.5$ then $p_{\max}\le 1+1=2$. Then a sum of each positive and negative elements of $D$ is
non-positive and $p_{\min}\not\in D+D$, a contradiction. Thus $p_{\min}=0.5$ and one of numbers $2.5$,
$3.5$, and $4.5$ belongs to $D$. But $2.5+0.5+(-3)=0$ and $3.5+0.5+(-4)=0$, so
either $z(D)\le 3$ or $4.5$ belongs to $D$. We assume the last case.
Then $D'=\{-4,-3,-2,0.5,1,4.5\}\subset D$. Let $p$ be a unique element of $D\setminus D'$. Since
$D'+D'\not\ni 4.5$, $4.5=p+s$ for some $s\in D$. Since $p\ge 4.5/2=2.25$, $p\not\in D'+D'$
and hence $p=4.5+s'$ for some $s'\in D$. Then $s+s'=0$ and so $z(D)\le 2$.
\end{itemize}
\end{proof}

\section{Proof of Theorem~\ref{t:main}}\label{s:tree}

First we introduce some notation. For a function $f:X\to Y$ and subset $A\subset X$ put $f[A]=\{f(a):a\in A\}$.

By a {\em tree} we understand any non-empty finite partially ordered set $(T,\le)$ such that for every $x\in T$ the set ${\downarrow}x=\{t\in T:t\le x\}$ is linearly ordered.

Let $T$ be a tree. By $\min T$ we denote the smallest element of $T$ and by $\max T$ the set of all maximal elements of $T$. A {\em branch} in a tree is a maximal linearly ordered subset $B\subset T$, which can be identified with the largest element of $B$.

For an element $x$ of a tree $T$ let ${\uparrow}x:=\{y\in T:x\le y\}$  and $\suc_T(x):=\min({\uparrow}x\setminus\{x\})$ be the set of immediate successors of $x$ in the tree $T$.  For any $x\in \max T$ we have $\suc_T(x)=\emptyset$.
For any element $x\ne\min T$ let $\pre_T(x)$ be the unique element $y\in T$ such that $x\in\suc_T(y)$.

A tree $T$ is called {\em binary} if for each $x\in T\setminus\max T$ the set
$\suc_T(x)$ has cardinality 2.

For a branch $B$ in a binary tree, let $\perp_B:B\setminus\max B\to T$ be the function assigning to each element $x\in B\setminus \max B$ the unique element of the set $\suc_T(x)\setminus B$.

A function $f:\max T\to T$ is called {\em regressive} if $f(x)<x$ for each $x\in \max T$.

\begin{lemma}\label{regres} For any regressive function $f:\max T\to T$ on a binary tree $T$ there are distinct elements $x,y\in \max T$ such that $f(x)=f(y)=\max({\downarrow}x\cap{\downarrow}y)$.
\end{lemma}

\begin{proof} The proof if by induction on the height $\hbar(T):=\max\{|{\downarrow}x|:x\in T\}$ of the binary tree $T$. If $\hbar(T)=1$, then no regressive function $f:\max T\to T$ exists, so the statement of the lemma holds.

Assume that the lemma has been proved for all binary trees of height $<n$. Take a binary tree $T$ of height $n>1$. Let $\min T$ be the root of $T$ and $x_1,x_2$ be two immediate successors of $\min T$ in $T$. Then $T_1:={\uparrow} x_1\setminus\{\min T\}$ and $T_2:={\uparrow}x_2\setminus\{\min T\}$ are trees of height $<n$. Two cases are possible.

1. $f(\max T_i)\subset T_i$ for some $i\in\{1,2\}$. In this case we can apply the inductive assumption and find two elements $x,y\in\max T_i$ such that $f(x)=f(y)=\max({\downarrow}x\cap{\downarrow}y)$.

2. For every $i\in\{1,2\}$ there exists an element $t_i\in\max T_i$ such that $f(t_i)=\min T$. Then $x=t_1$ and $y=t_2$ are two elements with $f(x)=f(y)=\max({\downarrow}x\cap{\downarrow}y)$.
\end{proof}

Now we can present the {\em proof of Theorem~\ref{t:main}}. Given a  finite decomposable subset $D$ of an Abelian group, we should find  two non-empty subsets $A,B\subset D$  with $\sum A+\sum B=0$.
\smallskip

 Let $D$ be a finite subset of an abelian group with $D\subset D+D$.
By a {\em binary $D$-tree} we understand a pair $(T,d)$ consisting of a binary tree $T$ and a function $d:T\to D$ such that each non-maximal element $x\in T$ we have $d(x)=d(y)+d(z)$, where $\{y,z\}=\suc_T(x)$. A binary $D$-tree is called {\em $\perp$-injective} if for each branch $L\subset T$ the restriction $d{\restriction}{\perp}_L[L\setminus\max L]$ is injective. This implies that $\big|{\perp}_L[L\setminus\max L]\big|\le |D|$ and hence $|L|\le |D|+1$. Consequently, each $\perp$-injective binary $D$-tree is finite, so we can choose a maximal $\perp$-injective binary $D$-tree $T$. Since $D\subset D+D$, the tree $T$ is a subtree of a binary $D$-tree $\tilde T=T\cup\max\tilde T$ such that $T\cap\max\tilde T=\emptyset$. The maximality of the tree $T$ ensures that for any $x\in\max T$, there exists an element $x'\in \suc_{\tilde T}(x)$ such that  $d(x')\in d[{\perp}_{{\downarrow}x}[{\downarrow}x\setminus\{x\}]$.
Let $M_2=\{x\in \max T:d[\suc_{\tilde T}(x)]\subset d[{\perp}_{{\downarrow}x}[{\downarrow}x\setminus\{x\}]\}$ and $M_1=\max T\setminus M_2$.

For every $x\in M_1$ let $x_1$ be the unique immediate successor of $x$ such that $$d(x_1)\in d\big[{\perp}_{{\downarrow}x}[{\downarrow}x\setminus\{x\}]\big]$$ and $g(x)\in{\downarrow}x$ be a (unique) point such that $d(x_1)=d({\perp}_{{\downarrow}x}(g(x)))$. Let $f(x):=g(x)$.

For every $x\in M_2$ and every point $x'\in \suc_{\tilde T}(x)$ there is a point $g(x')\in{\downarrow}x$ such that $d(x')=d({\perp_{{\downarrow}x}}(g(x')))$. Let $f(x):=\max\{g(x'):x'\in\suc_{\tilde T}(x)\}$.

Now observe that we have defined a regressive function $f:\max T\to T$. By Lemma~\ref{regres}, there are two maximal elements $x,y\in\max T$ such that $f(x)=f(y)=\max({\downarrow}x\cap{\downarrow}y)$.

By the definition of $f(x)$ the set $\suc_{\tilde T}(x)$ contains a point $x_1$ such that $f(x)=g(x_1)$. Let $x_2$ be the unique point of $\suc_{\tilde T}(x)\setminus\{x_1\}$. The definition of the function $f$ guarantees that $d(x_2)\notin \{d(\perp_{{\downarrow}x}(t)):f(x)<t<x\}$.
By analogy the set $\suc_{\tilde T}(y)$ can be written as $\{y_1,y_2\}$ such that and $f(y)=g(y_1)$ and  $d(y_2)\notin \{d(\perp_{{\downarrow}y}(t)):f(y)<t<y\}$.

Let $g'(x_1)={\perp}_{{\downarrow}x}(g(x_1))={\perp}_{{\downarrow}x}(f(x))$ and $g'(y_1)={\perp}_{{\downarrow}y}(g(y_1))={\perp}_{{\downarrow}y}(f(y))$. Observe that $\{g'(x_1),g'(y_1)\}=\suc_T(f(x))$, $g'(x_1)\in{\downarrow}y$ and $g'(y_1)\in{\downarrow}x$.

Let $$A_T=\{{\perp}_{{\downarrow}x}(t):f(x)<t<x\}\mbox{  and }B_T:=\{{\perp}_{{\downarrow}y}(t):f(y)<t<y\}.$$
It follows from the definition of $f(x)=f(y)$ that $d(x_2)\notin d[A_T]$ and $d(y_2)\notin d[B_T]$.

By induction it can be shown that
$$
d(y_1)=d(g'(y_1))=d(x_1)+d(x_2)+\sum_{t\in A_T}d(t)$$
and
$$d(x_1)=d(g'(x_1))=d(y_1)+d(y_2)+\sum_{t\in B_T}d(t).$$
Then $$d(y_1)-d(x_1)=\sum_{t\in A_T\cup \{x_2\}}d(t)=-\sum_{t\in B_T\cup \{y_2\}}d(t)$$ and finally $\sum A=-\sum B$ for the sets $A=d[A_T]\cup \{d(x_2)\}$ and $B=d[B_T]\cup\{d(y_2)\}$.

{}

\end{document}